\documentclass{amsart}
\usepackage{epic,latexsym,amssymb,xcolor}
\usepackage{color}
\usepackage{tikz}
\usepackage{amsfonts,epsf,amsmath,leftidx}
\usepackage{float}
\usepackage{pgfplots}
\usepackage{mathrsfs}
\usetikzlibrary{arrows}

\usepackage[bottom=1.5in, top=1.2in, left=1.5in, right=1.5in]{geometry}

\newtheorem{thm}{Theorem}
\newtheorem{lem}{Lemma}

\newtheorem{cor}{Corollary}

\newtheorem{prop}{Proposition}

\let\oldenumerate\enumerate
\renewcommand{\enumerate}{
  \oldenumerate
  \setlength{\itemsep}{0pt}
  \setlength{\parskip}{0pt}
  \setlength{\parsep}{0pt}
}

\begin{document}

\title {Quasi-strongly regular graphs of grade three with diameter two}

\author{Songpon Sriwongsa}
\address{Songpon Sriwongsa\\ Mathematics and Statistics with Applications (MaSA), Department of Mathematics, Faculty of Science, King Mongkut's University of Technology Thonburi (KMUTT), Bangkok 10140, Thailand}
\email{\tt songpon.sri@kmutt.ac.th, songponsriwongsa@gmail.com}

\author{Pawaton Kaemawichanurat}
\address{Pawaton Kaemawichanurat\\ Mathematics and Statistics with Applications (MaSA), Department of Mathematics, Faculty of Science, King Mongkut's University of Technology Thonburi (KMUTT), Bangkok 10140, Thailand}
\email{\tt pawaton.kae@kmutt.ac.th}



\begin{abstract}
%A graph $G$ is $(n, k, a; c_{1}, ..., c_{p})$-quasi-strongly regular, or $QSR(n, k, a; c_{1}, ..., c_{p})$ if $G$ is a $k$-regular graph of order $n$ which every pair of adjacent vertices share $a$ common neighbours and every pair of non-adjacent vertices share $c_{i}$ common neighbour for some $1 \leq i \leq p$. We say that $p$ is the grade of $G$. It was shown by Goldberg in 2006 that, when $p = 2$, every vertex $u$ of the $QSR(n, k, a; c_{1}, c_{2})$ graphs has the same number of vertices that share $c_{1}$ and $c_{2}$ common neighbours with $u$. However, when $p \geq 3$, Goldberg further gave an example of $QSR(n, k, a; c_{1}, c_{2}, c_{3})$ graph that different vertices $u$ and $v$ of $G$ have different numbers of vertices that share $c_{1}, c_{2}$ and $c_{3}$ common neighbours with $u$ and $v$ and the problem becomes very challenging. In this paper, we say that the $QSR(n, k, a; c_{1}, ..., c_{p})$ graph is strictly if, for all $1 \leq i \leq p$, there exists a pair of non-adjacent vertices that share $c_{i}$ common neighbours. We prove that if $G$ is a strictly $QSR(n, k, 0; k - 1, k - 2, k - 3)$ graph of diameter two, then
%\begin{center}
%$2k + 3 \leq n \leq k^{2} - 4.$
%\end{center}
%\noindent All the extremal graphs are characterized. Further, by our result, we characterize all strictly $QSR(n, k, 0; k - 1, k - 2, k - 3)$ graphs when $k = 4$.

A quasi-strongly regular graph of grade $p$ with parameters $(n, k, a; c_1, \ldots, c_p)$ is a $k$-regular graph of order $n$ such that any two adjacent vertices share $a$ common neighbours and any two non-adjacent vertices share $c_{i}$ common neighbours for some $1 \leq i \leq p$. This is a generalization of a strongly regular graph. In this paper, we focus on strictly quasi-strongly regular graphs of grade $3$ with $c_i = k - i$ for $i = 1, 2, 3$. The main result is to show the sharp bounds of order $n$ for a given $k \geq 4$. Furthermore, by this result, we characterize all of these graphs whose $n$ satisfies  upper or lower bounds.
\end{abstract}

\date{}

 \maketitle

\section{Introduction}

Let $G = (V(G), E(G))$ be a graph where $V(G)$ is the set of all vertices and $E(G)$ is the set of all edges. The \emph{order} of $G$ is $|V(G)|$. For a vertex $u \in V(G)$, a \emph{neighbour} of $u$ in $G$ is a vertex $v \in V(G) \setminus \{u\}$ that is adjacent to $u$. The \emph{neighbour set} of $u$ in $G$ is the set of all neighbours of $u$ in $G$ and is denoted by $N_{G}(u)$. The \emph{closed neighbour set} of $u$ in $G$ is $N_{G}[u] = N_{G}(u) \cup \{u\}$. The \emph{degree} of $u$ in $G$, denoted $deg_{G}(u)$, is the number of neighbours of $u$ in $G$, i.e., $deg_{G}(u) = |N_{G}(u)|$. In particular, for a subset $X$ of $V(G)$, a neighbour of $u$ in $X$ is a vertex $v$ in $X \setminus \{u\}$ that is adjacent to $u$. Hence, the neighbour set of $u$ in $X$ is the set of vertices in $X$ that are adjacent to $u$ and is denoted by $N_{X}(u)$, and the \emph{degree} of $u$ in $X$ is $deg_{X}(u) = |N_{X}(u)|$. Clearly, $N_{X}(u) = N_{G}(u) \cap X$. Similarly, for a subgraph $H$ of $G$, we denote $N_{V(H)}(u)$ and $deg_{V(H)}(u)$ by $N_{H}(u)$ and $deg_{H}(u)$, respectively. Further, for subsets $X$ and $Y$ of $V(G)$, we let $N_{Y}(X) = \{v \in Y : vu \in E(G)$ for some $u \in X\}$. In particular, if $x \in X$, the \emph{private neighbour set of} $x$ \emph{in} $Y$ \emph{with respect to} $X$ is denoted by $PN_{Y}(x, X)$ which is the set $\{v \in Y : N_{X}(v) = \{x\}\}$. The subgraph of $G$ \emph{induced} by $X$ is denoted by $G[X]$. A subset $X$ of $V(G)$ is \emph{independent} if $G[X]$ has no edge. The \emph{independence number} of $G$ is the maximum cardinality of an independent set of $G$ and is denoted by $\alpha(G)$. When no ambiguity can occur, we denote $N_{G}(u), N_{G}[u]$ and $deg_{G}(u)$ by $N(u), N[u]$ and $deg(u)$, respectively. For vertices $u, v \in V(G)$, the \emph{distance} between $u$ and $v$ is the length of a shortest path in $G$ starting from $u$ to $v$. The \emph{diameter} of $G$ is the maximum distance of any pair of vertices of $G$.
\vskip 5 pt

A graph $G$ is called $k$\emph{-regular} if every vertex of $G$ has degree $k$. A $k$-regular graph $G$ of order $n$ is said to be $(n, k, \mu, \lambda)$\emph{-strongly regular} if any two adjacent vertices share $\mu$ common neighbours and any two non-adjacent vertices share $\lambda$ common neighbours, while an $(n, k, a, b)$\emph{-Deza graph} is a $k$-regular graph of order $n$ which any two vertices share either $a$ or $b$ common neighbours where the number of common neighbours between any two vertices does not depend on their adjacency. Furthermore, an $(n, k, a; c_{1}, ..., c_{p})$\emph{-quasi-strongly regular graph}, \emph{or a} $QSR(n, k, a; c_{1}, ..., c_{p})$ \emph{graph}, is a $k$-regular graph of order $n$ such that any two adjacent vertices share $a$ common neighbours and any two non-adjacent vertices share $c_{i}$ common neighbours for some $1 \leq i \leq p$. The {\it grade} of this graph is the number of indices $1 \leq i \leq p$ for which there exist two non-adjacent vertices with $c_i$ common neighbours. Moreover, if the graph is of grade $p$, then it is said to be {\it proper}. A proper $QSR(n, k, a; c_{1}, ..., c_{p})$ is said to be \emph{strictly} quasi-strongly regular and denoted by $SQSR(n, k, a; c_{1}, ..., c_{p})$, if all $a, c_{1}, ..., c_{p}$ are distinct.
\vskip 5 pt

A concept of strongly regular graphs was introduced by Bose~\cite{B} in his classical paper in order to show the connection between these graphs with the concept of \emph{partial geometry} which plays an important role in area of coding theory in the part of coding techniques, the so called \emph{low density parity check code (LDPC)}. For example, see~\cite{DLLB,DDJ,JW,LLLA,LXDT,V}. Also, in~\cite{B}, Bose pointed out that the concept of strongly regular graphs is, up to isomorphism, the same as the condition list called \emph{association schemes} of \emph{partially balanced incomplete block design (PBIB)} in~\cite{BS} which has been extensively studied in area of combinatorial design.
\vskip 5 pt

 The research of strongly regular graphs has not only been enriched by a number of applications as mentioned in the above paragraph (for example), but it also has been varied and connected to many other mathematical theoretical concepts. One of very interesting connection was pointed out by Rotman~\cite{R} that a central simple Lie algebra over $\mathbb{Z}/ 2\mathbb{Z}$ can be constructed from a strongly regular graph. An example of a generalization of strongly regular graphs is the concept of \emph{Deza graphs} which was initiated by Erickson et. al.~\cite{EFHHH}. In fact the authors were inspired by Deza and Deza~\cite{DD} who provided relationship of some Deza graph in their work concerning \emph{metric polytope}. For more example of studies in Deza graphs see~\cite{GHKS,KMS,KS}. Another generalization of strongly regular graphs is the concept of \emph{quasi-strongly regular graphs} which was introduced in Golightly et. al.~\cite{GHS,GHS2} and was popularized  by Goldberg~\cite{G}. Interestingly, in~\cite{G}, the author refined an observation in \cite{N} to show connection between quasi-strongly regular graphs and distance regular graphs. For more details about the latter graphs see Brouwer et. al.~\cite{BCN} and Cameron\cite{C}.
 \vskip 5 pt

 In the study of strictly quasi-strongly regular graphs of grade $2$, Goldberg~\cite{G} established the following two equations.

\begin{equation}
t_{1} + t_{2} = n - k - 1\notag
\end{equation}
and
\begin{equation}
c_{1}t_{1} + c_{2}t_{2} = k(k - a - 1)\notag
\end{equation}
where $t_{i} = t_i(u)$ is the number of vertices in $V(G) \setminus N[u]$ that share $c_{i}$ common neighbours for a fixed vertex $u \in V(G)$. It is worth noting that the proofs of these equations are obtained from choosing an arbitrary vertex $u$ and counting the number of vertices in $V(G) \setminus N[u]$ and counting the number of edges between $N(u)$ and $V(G) \setminus N[u]$. As there are only two variables $t_{1}$ and $t_{2}$ and two equations, the values $t_{1}$ and $t_{2}$ do not depend on the choice of $u$. However, this is not always true for the graphs of grade $p \geq 3$. In the same paper, it was pointed out that there exists an example of this type of graph of grade $3$ with parameters $(12, 4, 0; 3, 2, 1)$ that the values of $t_{1}(u), t_{2}(u), t_{3}(u)$ are vary and depending on a vertex $u$. So, a new machinery may be required to study the graphs of higher grade. These motivate us to explore  strictly quasi-strongly regular graphs of grade $3$. In this work, we focus on the graphs $SQSR(n, k, 0; k - 1, k - 2, k - 3)$ for $k \geq 4$. Of course, these graphs are a generalization of the above example and they have diameter $2$. The main result here is to find the sharp bounds of $n$ for a given $k \geq 4$. Moreover, we can characterize these graphs whose $n$ is equal to the upper or lower bound. Some proper quasi-strongly regular graphs with particular parameters have been characterized in \cite{JYZ, XJZ}. According to our knowledge, these graphs of grade higher than 2 with any parameters have not been characterized yet. It can be shown in this paper that there is only one graph of $SQSR(11, 4, 0; 3, 2, 1)$ and only one graph of $SQSR(12, 4, 0; 3, 2, 1)$, up to isomorphism.

\section{Main results}
In this section, we illustrate the main theorem where the proof is provided in Section \ref{proofs}. The result is to establish upper and lower bounds of $n$ for $SQSR(n, k, 0; k - 1, k - 2, k - 3)$ graphs when $k \geq 4$. Further, we characterize all such graphs achieving the upper and lower bound for $n$.

Let us introduce quasi-strongly regular graphs $G_1$ of order $11$ and $G_2$ of order $12$ in the figures below.

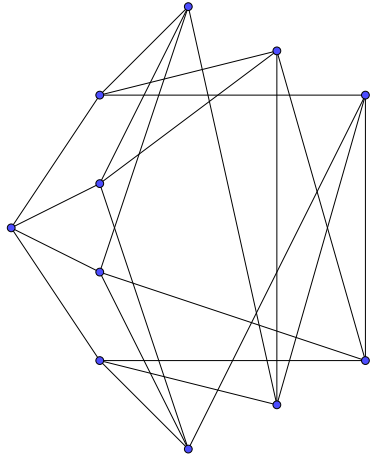
\begin{figure}[H]
\centering
\definecolor{ududff}{rgb}{0.30196078431372547,0.30196078431372547,1}
\resizebox{0.35\textwidth}{!}{%
\begin{tikzpicture}[line cap=round,line join=round,>=triangle 45,x=1cm,y=1cm]
%\clip(-9.963514871137317,-9.3747732545775) rectangle (15.801968014629482,6.293991947498556);
\draw [line width=0.4pt] (-6,-2)-- (-4,1);
\draw [line width=0.4pt] (-6,-2)-- (-4,-1);
\draw [line width=0.4pt] (-6,-2)-- (-4,-3);
\draw [line width=0.4pt] (-6,-2)-- (-4,-5);
\draw [line width=0.4pt] (-4,1)-- (-2,3);
\draw [line width=0.4pt] (-4,-1)-- (-2,3);
\draw [line width=0.4pt] (-4,-3)-- (-2,3);
\draw [line width=0.4pt] (-4,1)-- (0,2);
\draw [line width=0.4pt] (-4,-1)-- (0,2);
\draw [line width=0.4pt] (-4,-1)-- (-2,-7);
\draw [line width=0.4pt] (-4,1)-- (2,1);
\draw [line width=0.4pt] (-4,-3)-- (2,-5);
\draw [line width=0.4pt] (-4,-3)-- (-2,-7);
\draw [line width=0.4pt] (-4,-5)-- (-2,-7);
\draw [line width=0.4pt] (-2,-7)-- (2,1);
\draw [line width=0.4pt] (-4,-5)-- (0,-6);
\draw [line width=0.4pt] (-4,-5)-- (2,-5);
\draw [line width=0.4pt] (-2,3)-- (0,-6);
\draw [line width=0.4pt] (0,2)-- (0,-6);
\draw [line width=0.4pt] (0,2)-- (2,-5);
\draw [line width=0.4pt] (0,-6)-- (2,1);
\draw [line width=0.4pt] (2,1)-- (2,-5);
\begin{scriptsize}
\draw [fill=ududff] (-6,-2) circle (2.5pt);
\draw [fill=ududff] (-4,1) circle (2.5pt);
\draw [fill=ududff] (-4,-1) circle (2.5pt);
\draw [fill=ududff] (-4,-3) circle (2.5pt);
\draw [fill=ududff] (-4,-5) circle (2.5pt);
\draw [fill=ududff] (-2,3) circle (2.5pt);
\draw [fill=ududff] (0,2) circle (2.5pt);
\draw [fill=ududff] (-2,-7) circle (2.5pt);
\draw [fill=ududff] (0,-6) circle (2.5pt);
\draw [fill=ududff] (2,1) circle (2.5pt);
\draw [fill=ududff] (2,-5) circle (2.5pt);
\end{scriptsize}
\label{g1}
\end{tikzpicture}
 }%
\vskip -0.25 cm
\caption{The graph $G_{1}$.}
\label{g1}
\end{figure}

\begin{figure}[H]
\centering
\definecolor{ududff}{rgb}{0.30196078431372547,0.30196078431372547,1}
\resizebox{0.45\textwidth}{!}{%
\begin{tikzpicture}[line cap=round,line join=round,>=triangle 45,x=1cm,y=1cm]
%\clip(-10.467977787618807,-5.09652186265566) rectangle (11.29647916803348,8.236658975041195);
\draw [line width=0.4pt] (-1,4)-- (-4,2);
\draw [line width=0.4pt] (-1,4)-- (-2,2);
\draw [line width=0.4pt] (-1,4)-- (0,2);
\draw [line width=0.4pt] (-1,4)-- (2,2);
\draw [line width=0.4pt] (1,4)-- (-2,2);
\draw [line width=0.4pt] (1,4)-- (0,2);
\draw [line width=0.4pt] (1,4)-- (2,2);
\draw [line width=0.4pt] (1,4)-- (4,2);
\draw [shift={(0,2)},line width=0.4pt]  plot[domain=0:3.141592653589793,variable=\t]({1*4*cos(\t r)+0*4*sin(\t r)},{0*4*cos(\t r)+1*4*sin(\t r)});
\draw [line width=0.4pt] (-4,2)-- (-6,0);
\draw [line width=0.4pt] (-6,0)-- (-2,2);
\draw [line width=0.4pt] (-6,0)-- (6,0);
\draw [line width=0.4pt] (-6,0)-- (4,-2);
\draw [line width=0.4pt] (-4,-2)-- (-4,2);
\draw [line width=0.4pt] (-4,-2)-- (0,2);
\draw [line width=0.4pt] (-4,-2)-- (6,0);
\draw [line width=0.4pt] (-4,-2)-- (0,-2);
\draw [line width=0.4pt] (0,-2)-- (-2,2);
\draw [line width=0.4pt] (0,-2)-- (2,2);
\draw [line width=0.4pt] (0,-2)-- (4,-2);
\draw [line width=0.4pt] (4,-2)-- (4,2);
\draw [line width=0.4pt] (4,-2)-- (0,2);
\draw [line width=0.4pt] (6,0)-- (4,2);
\draw [line width=0.4pt] (6,0)-- (2,2);
\begin{scriptsize}
\draw [fill=ududff] (-1,4) circle (2.5pt);
\draw [fill=ududff] (1,4) circle (2.5pt);
\draw [fill=ududff] (-2,2) circle (2.5pt);
\draw [fill=ududff] (0,2) circle (2.5pt);
\draw [fill=ududff] (2,2) circle (2.5pt);
\draw [fill=ududff] (-4,2) circle (2.5pt);
\draw [fill=ududff] (4,2) circle (2.5pt);
\draw [fill=ududff] (-6,0) circle (2.5pt);
\draw [fill=ududff] (-4,-2) circle (2.5pt);
\draw [fill=ududff] (0,-2) circle (2.5pt);
\draw [fill=ududff] (4,-2) circle (2.5pt);
\draw [fill=ududff] (6,0) circle (2.5pt);
\end{scriptsize}
\end{tikzpicture}
 }%
\vskip -0.25 cm
\caption{The graph $G_{2}$.}
\label{g2}
\end{figure}
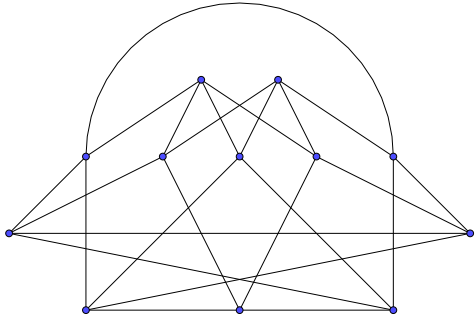

\begin{thm}[Main Theorem]\label{main ub}
Let $G$ be an $SQSR(n, k, 0; k - 1, k - 2, k - 3)$ graph when $k \geq 4$. Then
\begin{align}
2k + 3 \leq n \leq k^{2} - 4.\notag
\end{align}
Furthermore, $G$ achieves the upper bound if and only if $G$ is isomorphic to $G_{2}$ and $G$ achieves the lower bound if and only if $G$ is isomorphic to $G_{1}$.
\end{thm}
\vskip 5 pt

\noindent Note that the graphs $G_{1}$ and $G_{2}$ are $SQSR(11, 4, 0; 3, 2, 1)$ and $SQSR(12, 4, 0; 3, 2, 1)$, respectively. Thus we completely characterize all $SQSR(n, k, 0; k - 1, k - 2, k - 3)$ graphs when $k = 4$ as detailed in the following corollary.
\vskip 5 pt

\begin{cor}
Let $G$ be an $SQSR(n, k, 0; k - 1, k - 2, k - 3)$ graph. If $k = 4$, then $G$ is either $G_{1}$ or $G_{2}$.
\end{cor}
\vskip 5 pt

\noindent We further have the following corollary.
\vskip 5 pt

\begin{cor}
Let $G$ be an $SQSR(n, k, 0; k - 1, k - 2, k - 3)$ graph. If $k \geq 5$, then
\begin{center}
$2k + 4 \leq n \leq k^{2} - 5$.
\end{center}
\end{cor}
\vskip 5 pt

\section{Proof of the main theorem}\label{proofs}
Let $u \in V(G)$, $G' = G - N[u], N_{G}(u) = U = \{u_{1}, ..., u_{k}\}$ and $n' = |V(G')|$. Because $c_{3} = k - 3 > 0$, $V(G')$ can be partitioned into $T_{1}, T_{2}, T_{3}$ where $T_{i}$ is the set of vertices sharing $c_{i}$ neighbours with $u$ for all $i \in \{1, 2, 3\}$. We also let $t_{i} = |T_{i}|$. Throughout the proof, all $t_{1}, t_{2}, t_{3}$ depend on the vertex $u$ only. We first state the following equations in the first proposition. These were also presented in the proofs of the case of strictly strongly regular graphs of grade $2$ \cite{G}. For the completeness, we also give the detail of the proof here.

\begin{prop}
Under the above setting, we have
\begin{equation}\label{tool1}
t_{1} + t_{2} + t_{3} = n - k - 1
\end{equation}
\noindent and
\begin{equation}\label{tool2}
c_{1}t_{1} + c_{2}t_{2} + c_{3}t_{3} = k(k - 1).
\end{equation}
\end{prop}
\begin{proof}
Clearly, $n - k - 1 = n' = |V(G')| = |T_{1}| + |T_{2}| + |T_{3}| = t_{1} + t_{2} + t_{3}$ because $V(G')$ is partitioned by $T_{1}, T_{2}, T_{3}$. This proves Equation (\ref{tool1}). To prove Equation (\ref{tool2}), we count the number of edges between $U$ and $V(G')$. Since $a = 0$, the set $U$ is independent. Hence, every vertex in $U$ is adjacent to $k - 1$ vertices in $V(G')$. This implies that there are $k(k - 1)$ edges from $U$ to $V(G')$. On the other hand. Each vertex in $T_{i}$ share $c_{i}$ neighbours with $u$. That is, $deg_{U}(v) = c_{i}$ for all $v \in T_{i}$ and $i \in \{1, 2, 3\}$. Thus, there are $c_{1}t_{1} + c_{2}t_{2} + c_{3}t_{3}$ edges from $V(G')$ to $U$. By double counting, we have $c_{1}t_{1} + c_{2}t_{2} + c_{3}t_{3} = k(k - 1)$ and this proves Equation (\ref{tool2}).
\end{proof}

 In what follows, we separate the proof into two parts. The arguments of upper bound and the characterization of graphs satisfying the bound are given in Subsection \ref{ub}. While those for lower bound part are given in Subsection \ref{lb}.
\vskip 5 pt

\subsection{The Upper Bound}\label{ub}

Recall that $G' = G - N[u], N_{G}(u) = U = \{u_{1}, ..., u_{k}\}$ and $n' = |V(G')|$. We, further, let $U_{i} = N_{G'}(u_{i})$. Because $a = 0$, $U$ is an independent set of $G$. Hence, every vertex in $U$ is adjacent to $k - 1$ vertices in $G'$. Therefore
\begin{center}
$n' \leq k(k - 1)$
\end{center}
\noindent because $|U| = k$ and $c_{3} > 0$. That is
\begin{center}
$n = 1 + k + n' \leq k^{2} + 1$.
\end{center}
\noindent We establish the following Lemmas.
\vskip 5 pt

\begin{lem}\label{c0}
If $k^{2} - 4 \leq n \leq k^{2} + 1$, then $k = 4$.
\end{lem}
\begin{proof}
Assume that $k^{2} - 4 \leq n \leq k^{2} + 1$. Thus, $k^{2} - k - 5 \leq n' \leq k^{2} - k$. We suppose to the contrary that $k \geq 5$. We may let $S \subseteq V(G')$ be the set such that every vertex in $S$ is adjacent to at least two vertices in $U$ and $T \subseteq V(G')$ be the set such that every vertex in $T$ is adjacent to exactly one vertex in $U$, further, we let $U'_{i}$ be the set of private neighbours of $u_{i}$ in $V(G')$ with respect to $U$ for $i \in \{1, ..., k\}$. Clearly, $U'_{i} \subset U_{i}$. Also, it is worth noting that every vertex in $T$ is in exactly one set $U'_{i}$. Because $c_{3} > 0$, it follows that $U'_{1}, ..., U'_{k}, S$ partition $V(G')$ and $\cup^{k}_{i = 1}U'_{i} = T$. Thus, $|T| + |S| = n'$ and $|T| = \sum^{k}_{i = 1}|U'_{i}|$ which imply that
\begin{align}\label{u'}
|T| = \sum^{k}_{i = 1}|U'_{i}| = n' - |S|.
\end{align}
\vskip 5 pt

\indent For any $l \in \{0, 1, ..., 5\}$, we let $n' = k^{2} - k - l$. When $l = 5$, we have that $n = 1 + k + k^{2} - k - l = k^{2} - 4$ which is odd if $k = 5$. Thus, $k \geq 6$ and this implies that $k - 1 \geq l$. When $l \leq 4$, by the assumption that $k \geq 5$, we have $k - 1 \geq 4 \geq l$. In both cases,
\begin{align}\label{kl}
l \leq k - 1.
\end{align}
\noindent We next show that the upper bound of $|S|$ is $l$. It can be observed that, for each $u_{i} \in U$, the vertex $u_{i}$ is adjacent to $k - 1 - |U'_{i}|$ vertices in $S$. Thus, $deg_{S}(u_{i}) = k - 1 - |U'_{i}|$. By Equation (\ref{u'}),
\begin{align}\label{s}
\sum^{k}_{i = 1}deg_{S}(u_{i}) &= k(k - 1) - \sum^{k}_{i = 1}|U'_{i}|\notag\\
                               &= k(k - 1) - (n' - |S|) \notag\\
                               &= k(k - 1) - (k^{2} - k - l - |S|) = l + |S|.
\end{align}
\noindent Let $S = \{s_{1}, ..., s_{|S|}\}$. By the definition of $S$, $s_{i}$ is adjacent to at least two vertices in $U$. Thus $deg_{U}(s_{i}) \geq 2$ and this implies that
\begin{align}\label{l}
\sum^{|S|}_{i = 1}deg_{U}(s_{i}) \geq 2|S|.
\end{align}
\noindent Hence, by Equation (\ref{s}) and double counting, we have
\begin{align}
2|S| \leq \sum^{|S|}_{i = 1}deg_{U}(s_{i}) = \sum^{k}_{i = 1}deg_{S}(u_{i}) = l + |S|\notag
\end{align}
\noindent which implies that
\begin{align}\label{sl}
|S| \leq l.
\end{align}

\indent Since every vertex $s$ in $S$ is adjacent to at least two vertices in $U$, $s$ is adjacent to at most $k - 2$ vertices in $T$. By Equation (\ref{sl}),
\begin{align}\label{ss}
\sum_{s \in S}deg_{T}(s) \leq |S|(k - 2) \leq l(k - 2).
\end{align}
\noindent By Equation (\ref{kl}), we have $kl \leq k(k - 1)$ which implies that $l(k - 2) + 2l \leq k(k - 1)$. Thus, by Equation (\ref{sl}), we have
\begin{align}\label{kl2}
l(k - 2) \leq k(k - 1) - 2l \leq k(k - 1) - l - |S| = n' - |S| = |T|.
\end{align}
\noindent Therefore, by Equations (\ref{ss}) and (\ref{kl2}), we have
\begin{align}\label{st}
\sum_{s \in S}deg_{T}(s) \leq |T|.
\end{align}
\noindent If $\sum_{s \in S}deg_{T}(s) < |T|$, then there exists a vertex $v \in U'_{i}$ for some $i \in \{1, ..., k\}$ which is not adjacent to any vertex in $S$. Since $c_{3} > 0$ and $deg_{G'}(v) = k - 1$, $v$ is adjacent to exactly one vertex in each $U'_{j}$ for $j \in \{1, ..., k\} \setminus \{i\}$. Thus, $|N(v) \cap N(u_{j})| = 1$. But $c_{3} = k - 3$ is minimum among $c_{1}, c_{2}, c_{3}$, it follows that $|N(v) \cap N(u_{j})| = 1 = k - 3 = c_{3}$ which implies that $k = 4$ contradicting the assumption. Thus, we may assume the the equality in (\ref{st}) holds. This implies that Equations (\ref{kl}), (\ref{l}) and (\ref{sl}) holds. The equalities of (\ref{kl}) and (\ref{sl}) imply that $|S| = l = k - 1$. Also, the equality of (\ref{l}) implies that every vertex in $S$ is adjacent to exactly two vertices in $U$ and has $k - 2$ private neighbours in $T$ with respect to $S$. Thus, the equality of (\ref{st}) implies that every vertex in $T$ is adjacent to exactly one vertex in $S$. Therefore, we can let $x$ be a vertex in $U'_{i}$ such that $x$ is adjacent to exactly one vertex $s \in S$ which $s$ is adjacent to exactly two vertices $u_{j}$ and $u_{j'}$ in $U$. Because $a = 0$, $i \notin \{j, j'\}$. As $c_{3} > 0$, we have that $x$ shares at least one common neighbour with $u_{i'}$ for all $i' \in I = \{1, ..., k\} \setminus \{i, j, j'\}$. Since $deg_{G'}(x) = k - 1$ and $xs \in E(G)$, it follows that $x$ is adjacent to at most $k - 2$ vertices in $\cup_{i' \in I}U_{i'}$. Because $k \geq 5$, it follows that there exists $i'' \in I$ such that $|N(x) \cap N(u_{i''})| = 1$. Hence, $|N(x) \cap N(u_{i''})| = 1 = k - 3 = c_{3}$ which implies that $k = 4$ contradicting the assumption. Therefore, $k = 4$.
\end{proof}
%%%%%%%%%%%%%%%%%%%%%%%%%%%%%%%%%%%%%%%%%%%%%%%%%%%%%%%%%%%%%%%%%%%%%%%%%%%%%%%%%%%%%%%%%%%%%%%%%%%%%%%%
%\vskip 5 pt

%\indent By Claim \ref{c0}, we may assume in the following that if $k^{2} - 4 \leq n \leq k^{2} + 1$, then $k = 4$.
%\vskip 5 pt

\begin{lem}\label{c1}
If $k = 4$, then $n \leq k^{2} - 4$.
\end{lem}

\begin{proof}
Suppose first that $n \geq k^{2} - 2$. Thus, $n \geq 14$. Since $c_{1} = 3$, we may let $x$ and $y$ be a pair of non-adjacent vertices of $G$ such that $N(x) \cap N(y) = \{z_{1}, z_{2}, z_{3}\} = Z$. Further, we let $N(x) \setminus Z = \{x_{1}\}$ and $N(y) \setminus Z = \{y_{1}\}$. Clearly, $N(x) = \{x_{1}\} \cup Z$ and $N(y) = \{y_{1}\} \cup Z$ are independent sets because $a = 0$. Thus, to share a common neighbour between $x_{1}$ and $y$, $x_{1}y_{1} \in E(G)$. Let $H$ be the subgraph of $G$ induced by $V(G) \setminus (\{x, y, x_{1}, y_{1}\} \cup Z)$. By the assumption, $|V(H)| \geq 7$. Since all $z_{1}, z_{2}, z_{3}$ are adjacent to both $x$ and $y$ and $k = 4$, it follows that $|N_{H}(Z)| \leq 6$. So, there exists a vertex $p$ of $H$ which is not adjacent to any vertex in $Z$. To share a common neighbour with $x$ and $y$, we have $px_{1}, py_{1} \in E(G)$. Hence, $x_{1}x_{2}, px_{1}, py_{1} \in E(G)$ contradicting $a = 0$.
\vskip 5 pt

\indent So, we may assume that $n = k^{2} - 3$. Thus, $|V(H)| = 6$. Similarly, if $|N_{H}(Z)| \leq 5$, then there exists a vertex $p \in V(H)$ which is adjacent to both $x_{1}$ and $y_{1}$ in order to share a common neighbour with $x$ and $y$, respectively. This yields a contradiction. Thus, $|N_{H}(Z)| = 6$ which implies that each $z_{i}$ has exactly two private neighbours in $H$ with respect to $Z$. We let $PN_{H}(z_{i}, Z) = \{z'_{i}, z''_{i}\}$ for all $1 \leq i \leq 3$. The graph $G$ now is illustrated by Figure \ref{F}.

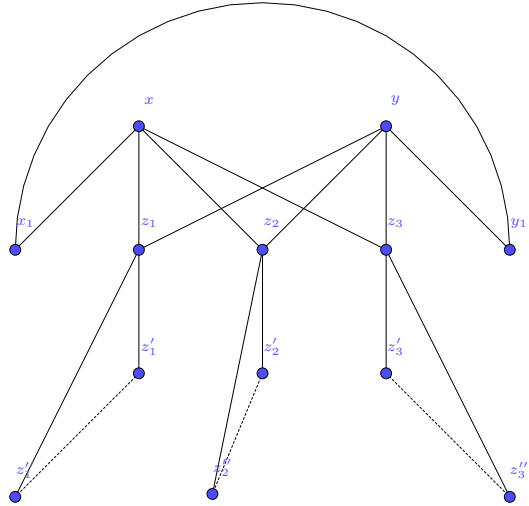
\begin{figure}[htb]
\centering
\definecolor{ududff}{rgb}{0.30196078431372547,0.30196078431372547,1}
\resizebox{0.5\textwidth}{!}{%
\begin{tikzpicture}[line cap=round,line join=round,>=triangle 45,x=1cm,y=1cm]
%\clip(-14.09527141918012,-9.999803940181332) rectangle (8.657775166761132,4.82535184926472);
\draw [line width=0.4pt] (-8,2)-- (-10,0);
\draw [line width=0.4pt] (-8,2)-- (-8,0);
\draw [line width=0.4pt] (-8,2)-- (-6,0);
\draw [line width=0.4pt] (-8,2)-- (-4,0);
\draw [line width=0.4pt] (-4,2)-- (-8,0);
\draw [line width=0.4pt] (-4,2)-- (-6,0);
\draw [line width=0.4pt] (-4,2)-- (-4,0);
\draw [line width=0.4pt] (-4,2)-- (-2,0);
\draw [line width=0.4pt] (-8,0)-- (-8,-2);
\draw [line width=0.4pt] (-8,0)-- (-10,-4);
\draw [line width=0.4pt] (-6,0)-- (-6,-2);
\draw [line width=0.4pt] (-6,0)-- (-6.810074909829786,-3.9528746041990757);
\draw [line width=0.4pt] (-4,0)-- (-4,-2);
\draw [line width=0.4pt] (-4,0)-- (-2,-4);
\draw [line width=0.4pt,dash pattern=on 1pt off 1pt] (-8,-2)-- (-10,-4);
\draw [line width=0.4pt,dash pattern=on 1pt off 1pt] (-6,-2)-- (-6.810074909829786,-3.9528746041990757);
\draw [line width=0.4pt,dash pattern=on 1pt off 1pt] (-4,-2)-- (-2,-4);
\draw [shift={(-6,0)},line width=0.4pt]  plot[domain=0:3.141592653589793,variable=\t]({1*4*cos(\t r)+0*4*sin(\t r)},{0*4*cos(\t r)+1*4*sin(\t r)});
\begin{scriptsize}
\draw [fill=ududff] (-8,2) circle (2.5pt);
\draw[color=ududff] (-7.8322376899489035,2.4172550198292986) node {$x$};
\draw [fill=ududff] (-10,0) circle (2.5pt);
\draw[color=ududff] (-9.834030116063945,0.4352823207054949) node {$x_1$};
\draw [fill=ududff] (-8,0) circle (2.5pt);
\draw[color=ududff] (-7.8322376899489035,0.4352823207054949) node {$z_1$};
\draw [fill=ududff] (-6,0) circle (2.5pt);
\draw[color=ududff] (-5.850264990825101,0.4352823207054949) node {$z_2$};
\draw [fill=ududff] (-4,0) circle (2.5pt);
\draw[color=ududff] (-3.8484725647100615,0.4352823207054949) node {$z_3$};
\draw [fill=ududff] (-4,2) circle (2.5pt);
\draw[color=ududff] (-3.8484725647100615,2.4172550198292986) node {$y$};
\draw [fill=ududff] (-2,0) circle (2.5pt);
\draw[color=ududff] (-1.8466801385950207,0.4352823207054949) node {$y_1$};
\draw [fill=ududff] (-8,-2) circle (2.5pt);
\draw[color=ududff] (-7.8322376899489035,-1.566510105409547) node {$z_1'$};
\draw [fill=ududff] (-6,-2) circle (2.5pt);
\draw[color=ududff] (-5.850264990825101,-1.566510105409547) node {$z_2'$};
\draw [fill=ududff] (-4,-2) circle (2.5pt);
\draw[color=ududff] (-3.8484725647100615,-1.566510105409547) node {$z_3'$};
\draw [fill=ududff] (-10,-4) circle (2.5pt);
\draw[color=ududff] (-9.834030116063945,-3.5683025315245893) node {$z_1''$};
\draw [fill=ududff] (-2,-4) circle (2.5pt);
\draw[color=ududff] (-1.8466801385950207,-3.5683025315245893) node {$z_3''$};
\draw [fill=ududff] (-6.810074909829786,-3.9528746041990757) circle (2.5pt);
\draw[color=ududff] (-6.643054070474623,-3.528663077542113) node {$z_2''$};
\end{scriptsize}
\end{tikzpicture}}
\vskip -0.25 cm
\caption{A subgraph of $G$.}
\label{F}
\end{figure}

%\vskip 10 pt

\noindent By $a = 0$, $\{z'_{i}, z''_{i}\}$ is independent. To share a common neighbour from $z'_{1}$ to both of $z_{2}$ and $z_{3}$, we have that $z'_{1}$ is adjacent to a vertex in $\{z'_{2}, z''_{2}\}$ and a vertex in $\{z'_{3}, z''_{3}\}$. Renaming vertices if necessary, we assume that $z'_{1}z'_{2}, z'_{1}z'_{3} \in E(G)$. Because $a = 0$, $z'_{2}z'_{3} \notin E(G)$. Thus to share a common neighbour between $z'_{2}$ and $z_{3}$ and between $z'_{3}$ and $z_{2}$, we have that $z'_{2}z''_{3}, z'_{3}z''_{2} \in E(G)$. Because $a = 0$, $z'_{1}z''_{2}, z'_{1}z''_{3} \notin E(G)$. Therefore, to share a common neighbour with $z_{1}$, we have that $z''_{2}z''_{1}, z''_{3}z''_{1} \in E(G)$. Thus $z''_{2}z''_{3} \notin E(G)$. The induced subgraph $H$ of $G$ is illustrated by Figure \ref{F1}.

\begin{figure}[htb]
\centering
\definecolor{ududff}{rgb}{0.30196078431372547,0.30196078431372547,1}
\resizebox{0.25\textwidth}{!}{%
\begin{tikzpicture}[line cap=round,line join=round,>=triangle 45,x=1cm,y=1cm]
%\clip(-4.623956308806709,-5.09770291555933) rectangle (16.366033182924955,8.14050293588692);
\draw [line width=0.4pt] (1,3)-- (3,3);
\draw [line width=0.4pt] (3,3)-- (5,1);
\draw [line width=0.4pt] (3,1)-- (5,3);
\draw [line width=0.4pt] (1,1)-- (3,1);
\draw [shift={(3,3)},line width=0.4pt]  plot[domain=0:3.141592653589793,variable=\t]({1*2*cos(\t r)+0*2*sin(\t r)},{0*2*cos(\t r)+1*2*sin(\t r)});
\draw [shift={(3,1)},line width=0.4pt]  plot[domain=3.141592653589793:6.283185307179586,variable=\t]({1*2*cos(\t r)+0*2*sin(\t r)},{0*2*cos(\t r)+1*2*sin(\t r)});
\draw [line width=0.4pt,dash pattern=on 1pt off 1pt] (1,3)-- (1,1);
\draw [line width=0.4pt,dash pattern=on 1pt off 1pt] (3,3)-- (3,1);
\draw [line width=0.4pt,dash pattern=on 1pt off 1pt] (5,3)-- (5,1);
\draw [line width=0.4pt,dash pattern=on 1pt off 1pt] (1,3)-- (3,1);
\draw [line width=0.4pt,dash pattern=on 1pt off 1pt] (1,3)-- (5,1);
\draw [line width=0.4pt,dash pattern=on 1pt off 1pt] (1,1)-- (3,3);
\draw [line width=0.4pt,dash pattern=on 1pt off 1pt] (1,1)-- (5,3);
\draw [line width=0.4pt,dash pattern=on 1pt off 1pt] (3,3)-- (5,3);
\draw [line width=0.4pt,dash pattern=on 1pt off 1pt] (3,1)-- (5,1);
\begin{scriptsize}
\draw [fill=ududff] (1,3) circle (2.5pt);
\draw[color=ududff] (1.1456360809947437,3.3885533755749973) node {$z_1'$};
\draw [fill=ududff] (3,3) circle (2.5pt);
\draw[color=ududff] (3.1455254676437137,3.3885533755749973) node {$z_2'$};
\draw [fill=ududff] (5,3) circle (2.5pt);
\draw[color=ududff] (5.145414854292683,3.3885533755749973) node {$z_3'$};
\draw [fill=ududff] (1,1) circle (2.5pt);
\draw[color=ududff] (1.1456360809947437,1.388663988926032) node {$z_1''$};
\draw [fill=ududff] (3,1) circle (2.5pt);
\draw[color=ududff] (3.1455254676437137,1.388663988926032) node {$z_2''$};
\draw [fill=ududff] (5,1) circle (2.5pt);
\draw[color=ududff] (5.145414854292683,1.388663988926032) node {$z_3''$};
\end{scriptsize}
\end{tikzpicture}}
\vskip -0.25 cm
\caption{The subgraph of $G$ induced by $H$.}
\label{F1}
\end{figure}
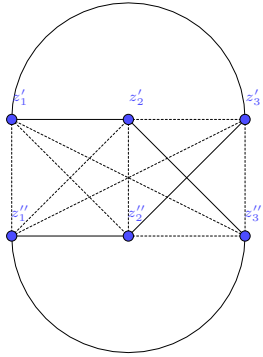

\noindent Since $x_{1}y_{1}, x_{1}x, y_{1}y \in E(G)$ and $k = 4$, it follows that each of $x_{1}$ and $y_{1}$ is adjacent to exactly two vertices in $H$. Thus, there is at least one vertex in $H$ has degree at most $3$ in $G$ contradicting $k = 4$. Therefore, $n \leq k^{2} - 4$.
\end{proof}

%%%%%%%%%%%%%%%%%%%%%%%%%%%%%%%%%%%%%%%%%%%%%%%%%%%%%%%%%%%%%%%%%%%%%%%%%%%%%%%%%%%%%%%%%%%%%%%%%%%%%%%%

By Lammas \ref{c0} and \ref{c1}, we can conclude that if $G$ is an $SQSR(n, k, 0;, k - 1, k - 2, k - 3)$ graph, then $n \leq k^{2} - 4$. This establishes the upper bound.
\vskip 5 pt

Now, we characterize all graphs whose $n$ satisfies the upper bound. Assume that $n = k^{2} - 4$. By Lemma \ref{c0}, we have that $k = 4$ and this implies that $n = 12$. Thus, to characterize all graphs which are $SQSR(k^{2} - 4, k, 0; k - 1, k - 2, k - 3)$, it suffices to characterize the graphs $SQSR(12, 4, 0; 3, 2, 1)$. Let $G$ be a graph $SQSR(12, 4, 0; 3, 2, 1)$. Recall that $x$ and $y$ are non-adjacent vertices of $G$ such that $N(x) \cap N(y) = \{z_{1}, z_{2}, z_{3}\} = Z$ and $N(x) \setminus Z = \{x_{1}\}, N(y) \setminus Z = \{y_{1}\}$. To share a common neighbour between $x_{1}$ and $y$, we have that $x_{1}y_{1} \in E(G)$. We, further, recall that $H$ is the subgraph of $G$ induced by $V(G) \setminus (\{x, y, x_{1}, y_{1}\} \cup Z)$. Thus, $|V(H)| = 5$ and we may let $V(H) = \{a_{1}, ..., a_{5}\}$. Because $x_{1}y_{1} \in E(G)$ and $a = 0$, it follows that $N_{H}(x_{1}) \cap N_{H}(y_{1}) = \emptyset$. Renaming vertices if necessary, we let
\begin{center}
$N_{H}(x_{1}) = \{a_{1}, a_{2}\}$ and $N_{H}(y_{1}) = \{a_{4}, a_{5}\}$.
\end{center}
\noindent Thus, $a_{1}a_{2}, a_{4}a_{5} \notin E(G)$ since $a = 0$. We have the following Lemma.
\vskip 5 pt

\begin{lem}\label{c2}
Under the above setting,
if there exists $i \in \{1, 2, 3\}$ such that $|N_{H}(z_{i}) \cap N_{H}(x_{1})| = 1$ and $|N_{H}(z_{i}) \cap N_{H}(y_{1})| = 1$, then $G$ is isomorphic to $G_{2}$.
\end{lem}

\begin{proof}
 Without loss of generality, we may assume that $i = 2$ which $N_{H}(z_{2}) \cap N_{H}(x_{1}) = \{a_{2}\}$ and $N_{H}(z_{2}) \cap N_{H}(y_{1}) = \{a_{4}\}$. Since $a = 0$, $a_{2}a_{4} \notin E(G)$. To share a common neighbour between $a_{5}$ and $z_{2}$, we have $a_{5}a_{2} \in E(G)$. Similarly, to share a common neighbour between $a_{1}$ and $z_{2}$, we have $a_{1}a_{4} \in E(G)$.
\vskip 5 pt

\indent We have that $a_{3}$ is adjacent to $z_{1}$ or $z_{3}$ in order to share a common neighbour with $x$ and $y$. Without loss of generality, we let $a_{3}z_{1} \in E(G)$. If $a_{3}z_{3} \notin E(G)$, then $a_{3}$ is adjacent to three vertices of the set $\{a_{1}, a_{2}, a_{4}, a_{5}\}$ since $deg(a_{3}) = k = 4$. This implies that $G[N(a_{3})]$ contains an edge $a_{1}a_{4}$ or $a_{2}a_{5}$ contradicting $a= 0$.
\vskip 5 pt

\indent Hence, $a_{3}z_{3} \in E(G)$. To share a common neighbour between $a_{1}$ and $y$ and between $a_{5}$ and $x$, we have that either $a_{1}z_{1}, a_{5}z_{3} \in E(G)$ or $a_{1}z_{3}, a_{5}z_{1} \in E(G)$. It can be seen that these two cases are symmetric. Thus, without loss of generality, we let $a_{1}z_{1}, a_{5}z_{3} \in E(G)$. To share a common neighbour between $a_{3}$ and $z_{2}$, we have that $a_{3}a_{2} \in E(G)$ or $a_{3}a_{4} \in E(G)$. Further, because $deg(a_{2}) = deg(a_{4}) = k = 4$, it follows that $a_{3}a_{2}, a_{3}a_{4} \in E(G)$. Finally, since $deg(a_{1}) = deg(a_{5}) = k = 4$, we obtain $a_{1}a_{5} \in E(G)$. Therefore, $G$ is isomorphic to $G_{2}$.
\end{proof}

\indent We also establish the following Lemma.
\vskip 5 pt

\begin{lem}\label{c3}
For all $i \in \{1, ..., 5\}$, each $a_{i}$ is adjacent to at least one vertex in $Z$.
\end{lem}

\begin{proof}
To share a common neighbour to both of $x$ and $y$, each $a_{i}$ is adjacent to at least one vertex in $Z$.
\end{proof}

\indent Clearly, $\sum^{3}_{i = 1}deg_{H}(z_{i}) = 6$. Since $|V(H)| = 5$, by Lemma \ref{c3}, it follows that there exists $i \in \{1, ..., ,5\}$ such that $a_{i}$ is adjacent to exactly two vertices in $Z$ and $a_{j}$ is adjacent to exactly one vertex in $Z$ for all $j \in \{1, ..., 5\} \setminus \{i\}$. We may distinguish two cases.
\vskip 5 pt

\noindent \textbf{Case 1:} $i \in \{1, 2, 4, 5\}$.
\vskip 5 pt

\noindent Without loss of generality, we let $i = 1$ and $a_{1}$ is adjacent to $z_{1}$ and $z_{2}$. If $z_{1}$ is adjacent to $a_{4}$ or $a_{5}$, then $|N_{H}(x_{1}) \cap N_{H}(z_{1})| = 1$ and $|N_{H}(y_{1}) \cap N_{H}(z_{1})| = 1$. By Lemma \ref{c2}, the graph $G$ is isomorphic to $G_{2}$. Note that the similar arguments work for $z_2$. Next, we assume otherwise that $z_{1}a_{4}, z_{1}a_{5} \notin E(G)$ and $z_{2}a_{4}, z_{2}a_{5} \notin E(G)$. Thus, $a_{4}z_{3}, a_{5}z_{3} \in E(G)$ by Lemma \ref{c3}. Further, $a_{2}z_{i}, a_{3}z_{j} \in E(G)$ where $\{i, j\} = \{1, 2\}$ by Lemma \ref{c3}. Thus, $a_{1}a_{3} \notin E(G)$ since $a = 0$. So, $a_{3}a_{2}, a_{3}a_{4}, a_{3}a_{5} \in E(G)$ because $deg(a_{3}) = k = 4$. Since $a = 0$, $a_{2}a_{4}, a_{2}a_{5} \notin E(G)$. This implies that $deg(a_{2}) \leq 3$ which is a contradiction.
\vskip 5 pt

\noindent \textbf{Case 2:} $i = 3$.
\vskip 5 pt

\noindent Without loss of generality, we let $a_{3}z_{1}, a_{3}z_{3} \in E(G)$. If the conditions in Lemma \ref{c2} holds, we are done. Thus, it remains to consider when $N_{H}(z_{2}) = \{a_{4}, a_{5}\}$ or $N_{H}(z_{2}) = \{a_{1}, a_{2}\}$. Due to the symmetry, we may assume that $N_{H}(z_{2}) = \{a_{4}, a_{5}\}$. By Lemma \ref{c3}, we have $a_{1}z_{1}, a_{2}z_{3} \in E(G)$ or $a_{1}z_{3}, a_{2}z_{1} \in E(G)$. Renaming vertices if necessary, we let $a_{1}z_{1}, a_{2}z_{3} \in E(G)$. Thus, $a_{1}a_{2}, a_{2}a_{3}, a_{1}a_{3} \notin E(G)$ because $a = 0$. Hence, $a_{3}a_{4}, a_{3}a_{5} \in E(G)$ as $deg(a_{3}) = 4$. So, there are at most two edges from $a_{4}, a_{5}$ to $a_{1}, a_{2}$. This yields $deg(a_{i}) \leq 3$ for some $i \in \{1, 2\}$ contradicting $k = 4$.

 Therefore, from Case 1 and Case 2, we can conclude that if $G$ satisfies the upper bound, then $G$ is $G_{2}$. These complete the subsection.

%%%%%%%%%%%%%%%%%%%%%%%%%%%%%%%%%%%%%%%%%%%%%%%%%%%%%%%%%%%%%%%%%%%%%%%%%%%%%%%%%%%%%%%%%%%%%%%%%%%%%%%%

\subsection{The Lower Bound}\label{lb}
\indent Recall again that $G' = G - N[u], N(u) = U = \{u_{1}, ..., u_{k}\}, n' = |V(G')|$ and $U_{i} = N_{G'}(u_{i})$. The following lemma is obvious under the conditions that $a = 0$ and $G$ is $k$-regular. So, we may state without the proof.
\vskip 5 pt

\begin{lem}\label{b1}
There are $k(k - 1)$ edges between $G[U]$ and $G'$.
\end{lem}
\vskip 5 pt

\indent Since $a = 0$, every vertex in $U$ is adjacent to $k - 1$ vertices in $G'$. Thus, $n' \geq k - 1$. We prove the lower bound by constructing the following three lemmas. They will immediately imply $n' \geq k + 2$.
\vskip 5 pt

\begin{lem}\label{n=k-1}
$n' \geq k$.
\end{lem}
\begin{proof} We assume to the contrary that $n' = k - 1$. Since we have by Lemma \ref{b1} that there are $k(k - 1)$ edges from $G'$ to $G[U]$, by the pigeonhole principle, there is a vertex $v$ in $G'$ such that $|N(u) \cap N(v)| = |U \cap N(v)| = k > c_{1}$, a contradiction. Hence, $n' \geq k$.
\end{proof}

\begin{lem}\label{n=k}
$n' \geq k + 1$.
\end{lem}
\begin{proof} We assume to the contrary that $n' \leq k$. By Lemma \ref{n=k-1}, we have $n' = k$. Let $V(G') = \{x_{1}, ..., x_{k}\}$. Lemma \ref{b1} yields that $\sum^{k}_{i = 1}deg_{U}(x_{i}) = k(k - 1)$. Clearly, $deg_{U}(x_{i}) = |N_{U}(x_{i})| \leq c_{1} = k - 1$ for all $i \in \{1, ..., k\}$. If there exists $x_{i}$ such that $deg_{U}(x_{i}) < k - 1$, then $k(k - 1) > \sum^{k}_{i = 1}deg_{U}(x_{i}) = k(k - 1)$, a contradiction. Thus,
\begin{align}
deg_{U}(x_{i}) = k - 1\notag
\end{align}
\noindent for all $i \in \{1, ..., n\}$. Because every vertex in $U$ is adjacent to $k - 1$ vertices in $G'$, the subgraph of $G$ containing edges between $U$ and $G'$ is complete bipartite minus a perfect matching. For $1 \leq i \leq k$, we let $x_{i}$ be the only vertex in $G'$ that $u_{i}$ is not adjacent. Because $deg(x_{i}) = k$, $deg_{G'}(x_{i}) = 1$ implies that $G'$ is a union of edges. Without loss of generality, we let $G'$ be the union of edges $x_{1}x_{2}, x_{3}x_{4}, ..., x_{k - 1}x_{k}$. Because $k \geq 4$, $u_{3}$ exists. We see that $G[N(u_{3})]$ has an edge $x_{1}x_{2}$ contradicting $a = 0$. Therefore, $n' \geq k + 1$.
\end{proof}

\begin{lem}\label{n=k+1}
$n' \geq k + 2$.
\end{lem}
\begin{proof} We assume to the contrary $n' \leq k + 1$. By Lemma \ref{n=k}, we have that $n' = k + 1$. By Equations (\ref{tool1}) and (\ref{tool2}), we have

\begin{equation}\label{eq1}
t_{1} + t_{2} + t_{3} = k + 1
\end{equation}
\noindent and
\begin{equation}\label{eq2}
(k - 1)t_{1} + (k - 2)t_{2} + (k - 3)t_{3} = k(k - 1).
\end{equation}

\noindent By Equations (\ref{eq1}) and (\ref{eq2}), we have that

\begin{equation}\label{eq3}
t_{1} + 2t_{2} + 3t_{3} = 2k
\end{equation}

\noindent Recall that $u_{1} \in U$ and $U_{1} = N_{G'}(u_{1})$. Because $a = 0$, $U_{1}$ is an independent set of size $k - 1$. Thus, there are two vertices in $V(G') \setminus U_{1}$ since $n' = k + 1$. Let $x$ be one of these vertices. Clearly, $|N_{G'}(x) \cap U_{1}| = |N(x) \cap N(u_{1})| \geq c_{3} = k - 3$. Moreover, $x$ shares at least $c_{3}$ common neighbours with $u$. That is $|N_{G'}(x)| \leq k - c_{3} = 3$. Therefore,
\begin{align}
3 \geq |N_{G'}(x)| \geq |N_{G'}(x) \cap U_{1}| \geq c_{3} = k - 3\notag
\end{align}
\noindent implying that
\begin{align}
k \leq 6.\notag
\end{align}
\noindent We then consider $G$ when $k = 4, 5$ and $6$.
\vskip 5 pt

\indent We first consider the case when $k = 6$. Because $U_{1}$ is an independent set of size $5$, $\alpha(G') \geq 5$. By Equations (\ref{eq1}) and (\ref{eq3}), we have
\begin{equation}
t_{1} + t_{2} + t_{3} = 7\notag
\end{equation}
\noindent and
\begin{equation}
t_{1} + 2t_{2} + 3t_{3} = 12\notag
\end{equation}
\noindent which can be solved that (\emph{i}) $t_{1} = 2, t_{2} = 5, t_{3} = 0$, (\emph{ii}) $t_{1} = 3, t_{2} = 3, t_{3} = 1$ and (\emph{iii}) $t_{1} = 4, t_{2} = 1, t_{3} = 2$. For the solution (\emph{i}) $t_{1} = 2, t_{2} = 5, t_{3} = 0$, there are $t_{1} = 2$ vertices of $G'$, each of which is adjacent to $c_{1} = k - 1$ vertices in $U$. Thus, all the $2$ vertices in $T_{1}$ have degree $1$ in $G'$. Similarly, there are $t_{2} = 5$ vertices of $G$, each of which is adjacent to $c_{2} = k - 2$ vertices in $U$. So, all the $5$ vertices in $T_{2}$ have degree $2$ in $G'$. Also, there is no vertex in $G'$ which is adjacent to $c_{3} = k - 3$ vertices in $U$ as $t_{3} = 0$. Thus, there is no vertex of degree $3$ in $G$. Therefore, the graph $G'$ has degree sequence $2, 2, 2, 2, 2, 1, 1$. Thus $G'$ is either $P_{7}$, the union of $C_{5}$ and $P_{2}$, the union of $C_{4}$ and $P_{3}$, or the union of $C_{3}$ and $P_{4}$. In each case, it can be checked that $\alpha(G') \leq 4$, a contradiction. For the solution (\emph{ii}) $t_{1} = 3, t_{2} = 3, t_{3} = 1$, the graph $G$ has degree sequence $3, 2, 2, 2, 1, 1, 1$. Thus $G$ is the union of $K_{1, 3}$ and $C_{3}$ or $G$ is one of the graphs in Figure \ref{f1}.

\begin{figure}[htb]
\centering
\definecolor{ududff}{rgb}{0.30196078431372547,0.30196078431372547,1}
\resizebox{0.7\textwidth}{!}{%
\begin{tikzpicture}[line cap=round,line join=round,>=triangle 45,x=1cm,y=1cm]
%\clip(-11.741135013826332,-7.01138685486012) rectangle (12.354502040507876,6.852872004095263);
\draw [line width=0.4pt] (-7,5)-- (-8,4);
\draw [line width=0.4pt] (-7,5)-- (-6,4);
\draw [line width=0.4pt] (-8,4)-- (-7,3);
\draw [line width=0.4pt] (-7,3)-- (-6,4);
\draw [line width=0.4pt] (-7,3)-- (-7,2);
\draw [line width=0.4pt] (-5,5)-- (-5,3);
\draw [line width=0.4pt] (-3,4)-- (-1,4);
\draw [line width=0.4pt] (2,5)-- (4,5);
\draw [line width=0.4pt] (-3,4)-- (-2,3);
\draw [line width=0.4pt] (-2,3)-- (-1,4);
\draw [line width=0.4pt] (-2,3)-- (-2,2);
\draw [line width=0.4pt] (0,5)-- (0,4);
\draw [line width=0.4pt] (0,4)-- (0.02,3.14);
\draw [line width=0.4pt] (2,5)-- (3,4);
\draw [line width=0.4pt] (3,4)-- (4,5);
\draw [line width=0.4pt] (3,4)-- (3,3);
\draw [line width=0.4pt] (3,3)-- (3,2);
\draw [line width=0.4pt] (5,5)-- (5,3);
\draw [line width=0.4pt] (-8,-2)-- (-7,-1);
\draw [line width=0.4pt] (-7,-1)-- (-7,-2);
\draw [line width=0.4pt] (-7,-1)-- (-6,-2);
\draw [line width=0.4pt] (-8,-2)-- (-8,-3);
\draw [line width=0.4pt] (-7,-2)-- (-7,-3);
\draw [line width=0.4pt] (-6,-2)-- (-6,-3);
\draw [line width=0.4pt] (-2,-1)-- (-3,-2);
\draw [line width=0.4pt] (-2,-1)-- (-2,-2);
\draw [line width=0.4pt] (-2,-1)-- (-1,-2);
\draw [line width=0.4pt] (-3,-2)-- (-3,-3);
\draw [line width=0.4pt] (-3,-3)-- (-3,-4);
\draw [line width=0.4pt] (-3,-4)-- (-3,-5);
\draw [line width=0.4pt] (3,-1)-- (2,-2);
\draw [line width=0.4pt] (3,-1)-- (3,-2);
\draw [line width=0.4pt] (3,-1)-- (4,-2);
\draw [line width=0.4pt] (2,-2)-- (2,-3);
\draw [line width=0.4pt] (3,-2)-- (3,-3);
\draw [line width=0.4pt] (2,-3)-- (2,-4);
\draw (-7.237104395208476,1.625972273847378) node[anchor=north west] {$H_1$};
\draw (-2.2511610355039364,1.625972273847378) node[anchor=north west] {$H_2$};
\draw (2.7903876404798367,1.625972273847378) node[anchor=north west] {$H_3$};
\draw (-7.255639500634888,-4.379401884309766) node[anchor=north west] {$H_4$};
\draw (-2.2326259300775253,-4.379401884309766) node[anchor=north west] {$H_5$};
\draw (2.771852535053426,-4.397936989736178) node[anchor=north west] {$H_6$};
\begin{scriptsize}
\draw [fill=ududff] (-7,5) circle (2.5pt);
\draw [fill=ududff] (-8,4) circle (2.5pt);
\draw [fill=ududff] (-6,4) circle (2.5pt);
\draw [fill=ududff] (-7,3) circle (2.5pt);
\draw [fill=ududff] (-7,2) circle (2.5pt);
\draw [fill=ududff] (-5,5) circle (2.5pt);
\draw [fill=ududff] (-5,3) circle (2.5pt);
\draw [fill=ududff] (-3,4) circle (2.5pt);
\draw [fill=ududff] (-1,4) circle (2.5pt);
\draw [fill=ududff] (-2,3) circle (2.5pt);
\draw [fill=ududff] (-2,2) circle (2.5pt);
\draw [fill=ududff] (0,5) circle (2.5pt);
\draw [fill=ududff] (0,4) circle (2.5pt);
\draw [fill=ududff] (0.02,3.14) circle (2.5pt);
\draw [fill=ududff] (2,5) circle (2.5pt);
\draw [fill=ududff] (4,5) circle (2.5pt);
\draw [fill=ududff] (3,4) circle (2.5pt);
\draw [fill=ududff] (3,3) circle (2.5pt);
\draw [fill=ududff] (3,2) circle (2.5pt);
\draw [fill=ududff] (5,5) circle (2.5pt);
\draw [fill=ududff] (5,3) circle (2.5pt);
\draw [fill=ududff] (-7,-1) circle (2.5pt);
\draw [fill=ududff] (-8,-2) circle (2.5pt);
\draw [fill=ududff] (-7,-2) circle (2.5pt);
\draw [fill=ududff] (-6,-2) circle (2.5pt);
\draw [fill=ududff] (-8,-3) circle (2.5pt);
\draw [fill=ududff] (-7,-3) circle (2.5pt);
\draw [fill=ududff] (-6,-3) circle (2.5pt);
\draw [fill=ududff] (-2,-1) circle (2.5pt);
\draw [fill=ududff] (-3,-2) circle (2.5pt);
\draw [fill=ududff] (-2,-2) circle (2.5pt);
\draw [fill=ududff] (-1,-2) circle (2.5pt);
\draw [fill=ududff] (-3,-3) circle (2.5pt);
\draw [fill=ududff] (-3,-4) circle (2.5pt);
\draw [fill=ududff] (-3,-5) circle (2.5pt);
\draw [fill=ududff] (3,-1) circle (2.5pt);
\draw [fill=ududff] (2,-2) circle (2.5pt);
\draw [fill=ududff] (3,-2) circle (2.5pt);
\draw [fill=ududff] (4,-2) circle (2.5pt);
\draw [fill=ududff] (2,-3) circle (2.5pt);
\draw [fill=ududff] (3,-3) circle (2.5pt);
\draw [fill=ududff] (2,-4) circle (2.5pt);
\end{scriptsize}
\end{tikzpicture}}
\vskip -0.25 cm
\caption{The graphs with degree sequence $3, 2, 2, 2, 1, 1, 1$.}
\label{f1}
\end{figure}
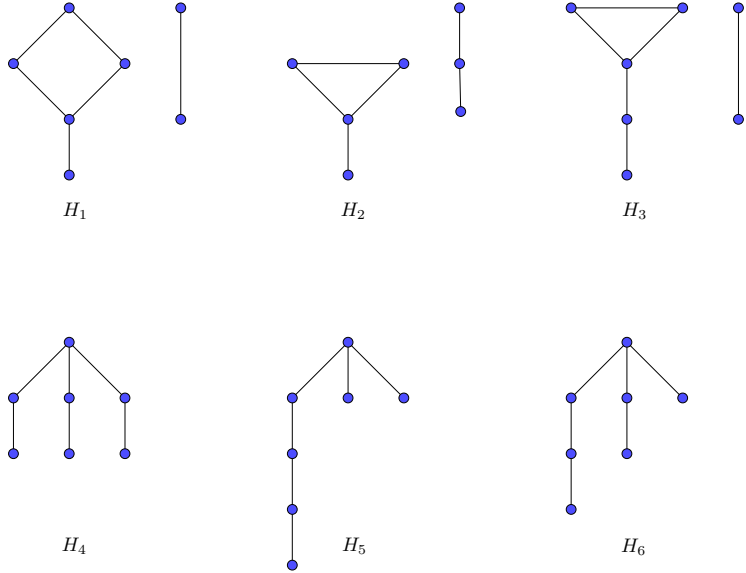

\noindent It is easy to check that $\alpha(G') \leq 4$ in all the cases, a contradiction. For the solution (\emph{iii}) $t_{1} = 4, t_{2} = 1, t_{3} = 2$, the graph $G'$ has degree sequence $3, 3, 2, 1, 1, 1, 1$. Thus, the graph $G'$ is one of the graphs in Figure \ref{f2}.

\begin{figure}[htb]
\centering
\definecolor{ududff}{rgb}{0.30196078431372547,0.30196078431372547,1}
\resizebox{0.7\textwidth}{!}{%
\begin{tikzpicture}[line cap=round,line join=round,>=triangle 45,x=1cm,y=1cm]
%\clip(-6.536201283908501,-11.817581748726594) rectangle (15.90906896967556,7.95691716458715);
\draw [line width=0.4pt] (-4,3)-- (-5,2);
\draw [line width=0.4pt] (-4,3)-- (-3,2);
\draw [line width=0.4pt] (-4,3)-- (-1,3);
\draw [line width=0.4pt] (-1,3)-- (-2,2);
\draw [line width=0.4pt] (-1,3)-- (0,2);
\draw [line width=0.4pt] (-5,2)-- (-5,1);
\draw [line width=0.4pt] (3,3)-- (2,2);
\draw [line width=0.4pt] (3,3)-- (4,2);
\draw [line width=0.4pt] (3,3)-- (5,3);
\draw [line width=0.4pt] (5,3)-- (7,3);
\draw [line width=0.4pt] (7,3)-- (6,2);
\draw [line width=0.4pt] (7,3)-- (8,2);
\draw (-2.818064524544543,0.6549951621884481) node[anchor=north west] {$H_7$};
\draw (4.781068471546902,0.6745805565082714) node[anchor=north west] {$H_8$};
\draw (2.450406547487928,-3.3012544904158427) node[anchor=north west] {$H_9$};
\draw [line width=0.4pt] (0,-2)-- (1,-1);
\draw [line width=0.4pt] (1,-1)-- (2,-2);
\draw [line width=0.4pt] (0,-2)-- (0,-3);
\draw [line width=0.4pt] (2,-2)-- (2,-3);
\draw [line width=0.4pt] (4,-1)-- (4,-3);
\draw [line width=0.4pt] (0,-2)-- (2,-2);
\begin{scriptsize}
\draw [fill=ududff] (-4,3) circle (2.5pt);
\draw [fill=ududff] (-1,3) circle (2.5pt);
\draw [fill=ududff] (-5,2) circle (2.5pt);
\draw [fill=ududff] (-3,2) circle (2.5pt);
\draw [fill=ududff] (-2,2) circle (2.5pt);
\draw [fill=ududff] (0,2) circle (2.5pt);
\draw [fill=ududff] (-5,1) circle (2.5pt);
\draw [fill=ududff] (2,2) circle (2.5pt);
\draw[color=ududff] (2.1566256326905786,2.42747334813245) node {$x$};
\draw [fill=ududff] (3,3) circle (2.5pt);
\draw[color=ududff] (3.1554807430015672,3.426328458443434) node {};
\draw [fill=ududff] (4,2) circle (2.5pt);
\draw[color=ududff] (4.154335853312556,2.42747334813245) node {$y$};
\draw [fill=ududff] (5,3) circle (2.5pt);
\draw[color=ududff] (5.153190963623546,3.426328458443434) node {$z$};
\draw [fill=ududff] (7,3) circle (2.5pt);
\draw[color=ududff] (7.150901184245523,3.426328458443434) node {};
\draw [fill=ududff] (6,2) circle (2.5pt);
\draw[color=ududff] (6.1520460739345335,2.42747334813245) node {$w$};
\draw [fill=ududff] (8,2) circle (2.5pt);
\draw[color=ududff] (8.149756294556513,2.42747334813245) node {$s$};
\draw [fill=ududff] (1,-1) circle (2.5pt);
\draw [fill=ududff] (0,-2) circle (2.5pt);
\draw [fill=ududff] (2,-2) circle (2.5pt);
\draw [fill=ududff] (0,-3) circle (2.5pt);
\draw [fill=ududff] (2,-3) circle (2.5pt);
\draw [fill=ududff] (4,-1) circle (2.5pt);
\draw [fill=ududff] (4,-3) circle (2.5pt);
\end{scriptsize}
\end{tikzpicture}}
\vskip -0.25 cm
\caption{The graphs with degree sequence $3, 3, 2, 1, 1, 1, 1$.}
\label{f2}
\end{figure}
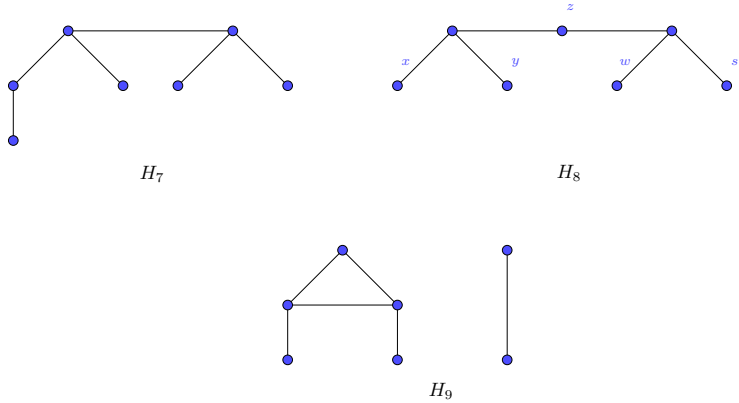

\noindent Clearly, the only one possible graph whose independence number is $5$ is when $G' = H_{8}$. However, $H_{8}$ has the unique maximum independent set $I_{H_{8}} = \{x, y, z, w, s\}$. Because $a = 0$ and each $U_{i}$ is an independent set containing $5$ vertices, it follows that $I_{H_{8}} = U_{1} = \cdots = U_{6}$. That is, all vertices in $U$ are adjacent to every vertex in $I_{H_{8}}$. Hence, $|N(u_{i}) \cap N(u_{j})| = k > c_{1}$ for any $u_{i}, u_{j} \in U$, a contradiction.
\vskip 5 pt

%\begin{figure}[htb]
%\begin{center}
%\includegraphics[width=0.8\textwidth]{h79}
%\end{center}
%\vskip -0.25 cm
%\caption{The graphs with degree sequence $3, 3, 2, 1, 1, 1, 1$.}
%\label{f3}
%\end{figure}
%\vskip 10 pt

 We then consider the case when $k = 5$. Clearly $U_{1}$ is an independent set of size $4$ and this implies $\alpha(G') \geq 4$. By Equations (\ref{eq1}) and (\ref{eq3}), we have
\begin{equation}
t_{1} + t_{2} + t_{3} = 6\notag
\end{equation}
\noindent and
\begin{equation}
t_{1} + 2t_{2} + 3t_{3} = 10\notag
\end{equation}
\noindent which can be solve that (\emph{i}) $t_{1} = 2, t_{2} = 4, t_{3} = 0$, (\emph{ii}) $t_{1} = 3, t_{2} = 2, t_{3} = 1$ and (\emph{iii}) $t_{1} = 4, t_{2} = 0, t_{3} = 2$. For the solution (\emph{i}) $t_{1} = 2, t_{2} = 4, t_{3} = 0$, the graph $G'$ has degree sequence $2, 2, 2, 2, 1, 1$. Thus $G'$ is either $P_{6}$, the union of $C_{3}$ and $P_{3}$, or the union of $C_{4}$ and $P_{2}$. In each case, it can be checked that $\alpha(G') \leq 3$, a contradiction. For the solution (\emph{ii}) $t_{1} = 3, t_{2} = 2, t_{3} = 1$, the graph $G'$ has degree sequence $3, 2, 2, 1, 1, 1$. Thus, $G'$ is one of the graphs in Figure \ref{f3}.

\begin{figure}[htb]
\centering
\definecolor{ududff}{rgb}{0.30196078431372547,0.30196078431372547,1}
\resizebox{0.65\textwidth}{!}{%
\begin{tikzpicture}[line cap=round,line join=round,>=triangle 45,x=1cm,y=1cm]
%\clip(-5.806435305673529,-6.333097232720378) rectangle (19.087716846284724,8.237189003324065);
\draw [line width=0.4pt] (0,3)-- (0,2);
\draw [line width=0.4pt] (0,3)-- (-1,2);
\draw [line width=0.4pt] (-1,2)-- (-1,1);
\draw [line width=0.4pt] (0,3)-- (1,2);
\draw [line width=0.4pt] (1,2)-- (1,1);
\draw [line width=0.4pt] (3,3)-- (5,3);
\draw [line width=0.4pt] (3,3)-- (4,2);
\draw [line width=0.4pt] (4,2)-- (5,3);
\draw [line width=0.4pt] (4,2)-- (4,1);
\draw [line width=0.4pt] (6,3)-- (6,1);
\draw [line width=0.4pt] (8,4)-- (9,3);
\draw [line width=0.4pt] (10,4)-- (9,3);
\draw [line width=0.4pt] (9,3)-- (9,2);
\draw [line width=0.4pt] (9,2)-- (9,1);
\draw [line width=0.4pt] (9,1)-- (9,0);
\draw (-0.3523174633039584,-0.33356760611384256) node[anchor=north west] {$H_{10}$};
\draw (3.6993129338848654,-0.33356760611384256) node[anchor=north west] {$H_{11}$};
\draw (8.763850930370895,-0.3530465984080196) node[anchor=north west] {$H_{12}$};
\begin{scriptsize}
\draw [fill=ududff] (0,3) circle (2.5pt);
\draw [fill=ududff] (-1,2) circle (2.5pt);
\draw [fill=ududff] (1,2) circle (2.5pt);
\draw [fill=ududff] (0,2) circle (2.5pt);
\draw [fill=ududff] (-1,1) circle (2.5pt);
\draw [fill=ududff] (1,1) circle (2.5pt);
\draw [fill=ududff] (3,3) circle (2.5pt);
\draw [fill=ududff] (5,3) circle (2.5pt);
\draw [fill=ududff] (4,2) circle (2.5pt);
\draw [fill=ududff] (4,1) circle (2.5pt);
\draw [fill=ududff] (6,3) circle (2.5pt);
\draw [fill=ududff] (6,1) circle (2.5pt);
\draw [fill=ududff] (8,4) circle (2.5pt);
\draw[color=ududff] (8.160002169251406,4.409567017518272) node {$z$};
\draw [fill=ududff] (10,4) circle (2.5pt);
\draw[color=ududff] (10.146859383257464,4.409567017518272) node {$w$};
\draw [fill=ududff] (9,3) circle (2.5pt);
\draw [fill=ududff] (9,2) circle (2.5pt);
\draw[color=ududff] (9.153430776254435,2.4227098035122117) node {$y$};
\draw [fill=ududff] (9,1) circle (2.5pt);
\draw [fill=ududff] (9,0) circle (2.5pt);
\draw[color=ududff] (9.153430776254435,0.4163735972119743) node {$u$};
\end{scriptsize}
\end{tikzpicture}}
\vskip -0.25 cm
\caption{The graphs with degree sequence $3, 2, 2, 1, 1, 1$.}
\label{f3}
\end{figure}
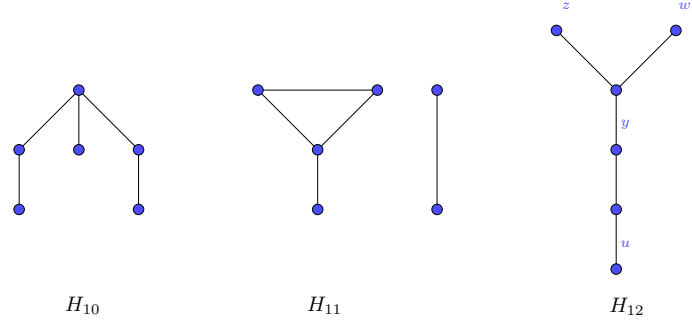

%\vskip 10 pt

%\begin{figure}[htb]
%\begin{center}
%\includegraphics[width=0.8\textwidth]{h1012}
%\end{center}
%\vskip -0.25 cm
%\caption{The graphs with degree sequence $3, 2, 2, 1, 1, 1$.}
%\label{f3}
%\end{figure}
%\vskip 10 pt

\noindent It can be checked that the only graph whose independence number is $4$ is $H_{12}$. However, $H_{12}$ has the unique maximum independent set $I_{H_{12}} = \{x, y, z, w\}$. Because $a = 0$ and every $U_{i}$ is an independent set of $4$ vertices, it follows that all vertices in $U$ are adjacent to every vertex in $I_{H_{12}}$. Hence, $|N(u_{i}) \cap N(u_{j})| = k > c_{1}$ for any $u_{i}, u_{j} \in U$, a contradiction. For the solution (\emph{iii}) $t_{1} = 4, t_{2} = 0, t_{3} = 2$, the graph $G'$ has degree sequence $3, 3, 1, 1, 1, 1$ which can be only the graph $H_{13}$ as illustrated in Figure \ref{f4}.

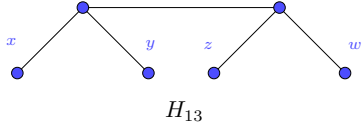
\begin{figure}[htb]
\centering
\definecolor{ududff}{rgb}{0.30196078431372547,0.30196078431372547,1}
\resizebox{0.35\textwidth}{!}{%
\begin{tikzpicture}[line cap=round,line join=round,>=triangle 45,x=1cm,y=1cm]
%\clip(-11.33,-7.33) rectangle (1.71,7.33);
\draw [line width=0.4pt] (-5,3)-- (-6,2);
\draw [line width=0.4pt] (-5,3)-- (-4,2);
\draw [line width=0.4pt] (-5,3)-- (-2,3);
\draw [line width=0.4pt] (-2,3)-- (-3,2);
\draw [line width=0.4pt] (-2,3)-- (-1,2);
\draw (-3.87,1.69) node[anchor=north west] {$H_{13}$};
\begin{scriptsize}
\draw [fill=ududff] (-5,3) circle (2.5pt);
\draw [fill=ududff] (-6,2) circle (2.5pt);
\draw[color=ududff] (-6.09,2.48) node {$x$};
\draw [fill=ududff] (-4,2) circle (2.5pt);
\draw[color=ududff] (-3.97,2.42) node {$y$};
\draw [fill=ududff] (-2,3) circle (2.5pt);
\draw [fill=ududff] (-3,2) circle (2.5pt);
\draw[color=ududff] (-3.09,2.42) node {$z$};
\draw [fill=ududff] (-1,2) circle (2.5pt);
\draw[color=ududff] (-0.85,2.42) node {$w$};
\end{scriptsize}
\end{tikzpicture}}
\vskip -0.25 cm
\caption{The graph with degree sequence $3, 3, 1, 1, 1, 1$.}
\label{f4}
\end{figure}

%\vskip 5 pt

\noindent Similarly, $H_{13}$ has the unique maximum independent set $I_{H_{13}} = \{x, y, z, w\}$. Because $a = 0$, all vertices in $U$ are adjacent to every vertex in $I_{H_{13}}$. Hence, $|N(u_{i}) \cap N(u_{j})| = k > c_{1}$ for any $u_{i}, u_{j} \in U$, a contradiction.

\indent We finally consider the case when $k = 4$. Clearly $U_{1}$ is an independent set of size $3$ and this implies $\alpha(G') \geq 3$. By Equations (\ref{eq1}) and (\ref{eq3}), we have
\begin{equation}
t_{1} + t_{2} + t_{3} = 5\notag
\end{equation}
\noindent and
\begin{equation}
t_{1} + 2t_{2} + 3t_{3} = 8\notag
\end{equation}
\noindent which can be solved that (\emph{i}) $t_{1} = 2, t_{2} = 3, t_{3} = 0$ and (\emph{ii}) $t_{1} = 3, t_{2} = 1, t_{3} = 1$. For the solution (\emph{i}) $t_{1} = 2, t_{2} = 3, t_{3} = 0$, the graph $G'$ has degree sequence $2, 2, 2, 1, 1$. Thus, $G'$ is either $P_{5}$, or the union of $C_{3}$ and $P_{2}$. Clearly, $\alpha(G') = 2$ when $G'$ is the union of $C_{3}$ and $P_{2}$. Hence, $G$ is $P_{5}$. However, it can be checked that $P_{5}$ has the unique maximum independent set, $I_{P_{5}}$ say. Because $a = 0$, all vertices in $U$ are adjacent to every vertex in $I_{P_{5}}$. Thus, $|N(u_{i}) \cap N(u_{j})| = k > c_{1}$ for any $u_{i}, u_{j} \in U$, a contradiction. For the solution (\emph{ii}) $t_{1} = 3, t_{2} = 1, t_{3} = 1$, the graph $G'$ has degree sequence $3, 2, 1, 1, 1$. Thus, the graph $G'$ is $H_{14}$ as illustrated in Figure \ref{f5}.

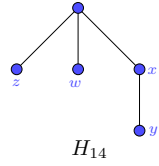
\begin{figure}[htb]
\centering
\definecolor{ududff}{rgb}{0.30196078431372547,0.30196078431372547,1}
\resizebox{0.15\textwidth}{!}{%
\begin{tikzpicture}[line cap=round,line join=round,>=triangle 45,x=1cm,y=1cm]
%\clip(-2.493884297520659,-6.966942148760331) rectangle (10.82842975206612,5.148760330578511);
\draw [line width=0.4pt] (4,3)-- (4,2);
\draw [line width=0.4pt] (4,3)-- (3,2);
\draw [line width=0.4pt] (4,3)-- (5,2);
\draw [line width=0.4pt] (5,2)-- (5,1);
\draw (3.7705785123966975,0.9669421487603295) node[anchor=north west] {$H_{14}$};
\begin{scriptsize}
\draw [fill=ududff] (4,3) circle (2.5pt);
\draw [fill=ududff] (3,2) circle (2.5pt);
\draw[color=ududff] (2.993719008264466,1.7851239669421475) node {$z$};
\draw [fill=ududff] (4,2) circle (2.5pt);
\draw[color=ududff] (3.9523966942148796,1.7685950413223128) node {$w$};
\draw [fill=ududff] (5,2) circle (2.5pt);
\draw[color=ududff] (5.208595041322317,1.983471074380164) node {$x$};
\draw [fill=ududff] (5,1) circle (2.5pt);
\draw[color=ududff] (5.225123966942152,0.9917355371900816) node {$y$};
\end{scriptsize}
\end{tikzpicture}}
\vskip -0.25 cm
\caption{The graph with degree sequence $3, 2, 1, 1, 1$.}
\label{f5}
\end{figure}

\noindent Similarly, $H_{14}$ has exactly two independent sets of size $3$ which are $\{x, w, z\}$ and $\{y, w, z\}$. Because $|U| = 4$, by the pigeonhole principle, there are two vertices $u_{i}, u_{j} \in U$ such that $N(u_{i}) \cap N(u_{j}) = \{x, w, z, u\}$ or $N(u_{i}) \cap N(u_{j}) = \{y, w, z, u\}$. In any case, we have $|N(u_{i}) \cap N(u_{j})| = k = 4 > c_{1}$, a contradiction. Hence, $n' \geq k + 2$ implying that $n \geq 2k + 3$. This proves the Lemma  and establishes the lower bound.
\end{proof}

Finally, we will show that if the graph $G$ is $SQSR(n, k, 0;, k - 1, k - 2, k - 3)$ when $n = 2k + 3$, then $G$ is isomorphic to $G_{1}$. Let $n = 2k + 3$. Because $|U_{1}| = k - 1$ and $n' = k + 2$, it follows that there are $3$ vertices in $V(G') \setminus U_{1}$. As $a = 0$, we have $U_{1}$ is an independent set. We let $x \in V(G') \setminus U_{1}$. Further, we let $X_{1} = N(x) \cap U_{1}, X_{2} = N(x) \cap U$ and $X_{3} = N(x) \cap (V(G') \setminus U_{1})$. Thus, $X_{3} = N(x) \setminus (X_{1} \cup X_{2})$ and
\begin{equation}\label{eq6}
k = |N(x)| = |X_{1}| + |X_{2}| + |X_{3}|.
\end{equation}
Clearly, $x$ has at least $c_{3}$ common neighbours with $u$ and $u_{1}$. Therefore,
\begin{equation}\label{eq7}
|X_{1}| \geq c_{3} = k - 3
\end{equation}
\noindent and
\begin{equation}\label{eq8}
|X_{2}| \geq c_{3} = k - 3.
\end{equation}
\noindent Hence, Equations (\ref{eq6}), (\ref{eq7}) and (\ref{eq8}) imply the equation
\begin{equation}
k \geq k - 3 + k - 3 + |X_{3}|
\end{equation}
\noindent which yields that
\begin{equation}\label{eq9}
6 \geq k + |X_{3}| \geq k
\end{equation}
\noindent because $|X_{3}| \geq 0$.
\vskip 5 pt

\begin{lem}\label{k=4}
If $n = 2k + 3$, then $k = 4$.
\end{lem}
\begin{proof} By Equation (\ref{eq9}), it suffice to show that $5 \leq k \leq 6$ is not possible. We first assume that $k = 6$. Thus, it must be equality throughout Equations (\ref{eq7}), (\ref{eq8}) and (\ref{eq9}). Therefore, $|X_{3}| = 0$, further, $x$ is adjacent to exactly three vertices in $U_{1}$ and is adjacent to exactly three vertices in $U$. Because $x$ is arbitrary vertex in $V(G') \setminus U_{1}$, this must be true for the other two vertices too. That is, for $y, z \in V(G') \setminus (U_{1} \cup \{x\})$, $y$ and $z$ are adjacent to exactly three vertices in $U_{1}$ and are adjacent to exactly three vertices in $U$. Hence, $V(G')$ is partitioned into two independent sets $U_{1}$ and $\{x, y, z\}$ and each vertex in $\{x, y, z\}$ is adjacent to exactly three vertices in $U_{1}$.
\vskip 5 pt

\noindent By Equation (\ref{tool1}), we have
\begin{equation}\label{eq4}
t_{1} + t_{2} + t_{3} = n - k - 1 = k + 2 = 8
\end{equation}
\noindent By Equations (\ref{eq4}) and (\ref{tool2}) we have that
\begin{equation}\label{eq5}
t_{1} + 2t_{2} + 3t_{3} = 3k = 18
\end{equation}
\noindent Further, Equations (\ref{eq4}) and (\ref{eq5}) yield
\begin{equation}\label{eq55}
t_{2} + 2t_{3} = 2k - 2 = 10.
\end{equation}
\noindent It can be checked that Equations (\ref{eq4}) and (\ref{eq55}) give \emph{(i)} $t_{1} = 0, t_{2} = 6, t_{3} = 2$, \emph{(ii)} $t_{1} = 1, t_{2} = 4, t_{3} = 3$ and \emph{(iii)} $t_{1} = 2, t_{2} = 2, t_{3} = 4$. The solution \emph{(i)} $t_{1} = 0, t_{2} = 6, t_{3} = 2$ gives that $G'$ has the degree sequences $3, 3, 2, 2, 2, 2, 2, 2$. This is not possible because all $x, y, z$ have degree $3$. Hence, $G'$ is the graph satisfying either \emph{(ii)} $t_{1} = 1, t_{2} = 4, t_{3} = 3$ and \emph{(iii)} $t_{1} = 2, t_{2} = 2, t_{3} = 4$. The solutions \emph{(ii)} and \emph{(iii)} give degree sequence $3, 3, 3, 2, 2, 2, 2, 1$ and $3, 3, 3, 3, 2, 2, 1, 1$ (see Figure \ref{h20}), respectively. From these two graphs, $U_{1}$ is the unique maximum independent set of $G'$ that contains $5$ vertices. Thus, $U_{1} = \cdots = U_{6}$. This yields a contradiction that $|N(u_{i}) \cap N(u_{j})| = |\{u\} \cup U_{1}| = k > k - 1 = c_{1}$. Hence, $k \neq 6$. Clearly, the case when $k = 5$ does not occur because $n = 2k + 3 = 2(5) + 3 = 13$ is odd number. Thus, $k = 4$.
\end{proof}

%%%%%%%%%%%%%%%%%%%%%%%%%%%%%%%%%%%%%%%%%%%%%%%%%%%%%%%%%%%%%%%%%%%%%%%%%%%%%%%%%%%%%%%%%%%%%%%%%%%%%%

\begin{figure}[htb]
\centering
\definecolor{ududff}{rgb}{0.30196078431372547,0.30196078431372547,1}
\resizebox{0.4\textwidth}{!}{%
\begin{tikzpicture}[line cap=round,line join=round,>=triangle 45,x=1cm,y=1cm]
%\clip(-9.412671555719031,-14.170088702203477) rectangle (37.26010403340945,13.984133991883636);
\draw [line width=0.4pt] (2,3)-- (4,3);
\draw [line width=0.4pt] (2,3)-- (4,2);
\draw [line width=0.4pt] (2,2)-- (4,3);
\draw [line width=0.4pt] (2,2)-- (4,1);
\draw [line width=0.4pt] (2,1)-- (4,2);
\draw [line width=0.4pt] (2,1)-- (4,1);
\draw [line width=0.4pt] (2,0)-- (4,3);
\draw [line width=0.4pt] (2,0)-- (4,2);
\draw [line width=0.4pt] (2,-1)-- (4,1);
\draw [line width=0.4pt] (7,3)-- (9,3);
\draw [line width=0.4pt] (7,3)-- (9,2);
\draw [line width=0.4pt] (7,3)-- (9,1);
\draw [line width=0.4pt] (7,2)-- (9,3);
\draw [line width=0.4pt] (7,2)-- (9,2);
\draw [line width=0.4pt] (7,1)-- (9,3);
\draw [line width=0.4pt] (7,1)-- (9,1);
\draw [line width=0.4pt] (9,2)-- (7,0);
\draw [line width=0.4pt] (7,-1)-- (9,1);
\draw [line width=0.4pt] (2,3)-- (4,3);
\draw (2.662601429578174,-1.3749662740360409) node[anchor=north west] {$H_{15}$};
\draw (7.65219168428471,-1.335985100171146) node[anchor=north west] {$H_{16}$};
\begin{scriptsize}
\draw [fill=ududff] (2,3) circle (2.5pt);
\draw [fill=ududff] (4,3) circle (2.5pt);
\draw [fill=ududff] (2,2) circle (2.5pt);
\draw [fill=ududff] (4,2) circle (2.5pt);
\draw [fill=ududff] (2,1) circle (2.5pt);
\draw [fill=ududff] (4,1) circle (2.5pt);
\draw [fill=ududff] (2,0) circle (2.5pt);
\draw [fill=ududff] (2,-1) circle (2.5pt);
\draw [fill=ududff] (7,3) circle (2.5pt);
\draw [fill=ududff] (9,3) circle (2.5pt);
\draw [fill=ududff] (7,2) circle (2.5pt);
\draw [fill=ududff] (9,2) circle (2.5pt);
\draw [fill=ududff] (7,1) circle (2.5pt);
\draw [fill=ududff] (9,1) circle (2.5pt);
\draw [fill=ududff] (7,0) circle (2.5pt);
\draw [fill=ududff] (7,-1) circle (2.5pt);
\end{scriptsize}
\end{tikzpicture}}
\vskip -0.25 cm
\caption{\footnotesize{The graph with degree sequences $3, 3, 3, 2, 2, 2, 2, 1$ (left) and $3, 3, 3, 3, 2, 2, 1, 1$ (right).}}
\label{h20}
\end{figure}
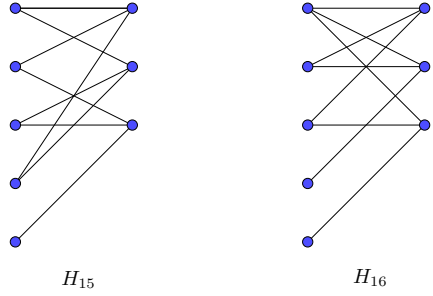

The above lemma implies that all graphs in the class of $SQSR(2k + 3, k, 0; k - 1, k - 2, k - 3)$ graphs are $SQSR(11, 4, 0; 3, 2, 1)$. Hence, it remains to prove that every $SQSR(11, 4, 0; 3, 2, 1)$ graph is $G_{1}$. Our approach here requires the following lemma.
\vskip 5 pt

\begin{lem}\label{b2}
If there exists a vertex $u \in V(G)$ which $N(u)$ contains $x$ and $y$ such that $N(x) \cap N(y) = \{u\}$, then $G$ is $G_{1}$.
\end{lem}

\begin{proof} We may let $N_{G'}(x) = N(x) \setminus \{u\} = \{x_{1}, x_{2}, x_{3}\}, N_{G'}(y) = N(y) \setminus \{u\} = \{y_{1}, y_{2}, y_{3}\}$ and $N(u) \setminus \{x, y\} = \{z, w\}$. Clearly, $V(G) \setminus  N[u] = \{x_{1}, x_{2}, x_{3}, y_{1}, y_{2}, y_{3}\}$ because $n = 11$. Since $a = 0$, the sets $\{x_{1}, x_{2}, x_{3}\}$ and $\{y_{1}, y_{2}, y_{3}\}$ are independent. Because $c_{i} \leq 3$ for all $i \in \{1, 2, 3\}$, it follows that $z$ is adjacent to exactly two vertices in one of the sets $N_{G'}(x)$ or $N_{G'}(y)$ and is adjacent to exactly one vertex in the other set. Renaming vertices if necessary, we may suppose that $z$ is adjacent to $x_{1}$ and $x_{2}$ in $N_{G'}(x)$ and is adjacent to $y_{1}$ in $N_{G'}(y)$. Since $a = 0$, $x_{1}y_{1}, x_{2}y_{1} \notin E(G)$. If neither $x_{1}$ nor $x_{2}$ is adjacent to $w$, then both $x_{1}$ and $x_{2}$ are adjacent to $y_{2}$ and $y_{3}$. This implies a contradiction that $|N(x_{1}) \cap N(x_{2})| = 4 > 3 = c_{1}$. Thus, at least one of $x_{1}$ or $x_{2}$ is adjacent to $w$. Renaming vertices if necessary, we may suppose that $x_{1}w \in E(G)$.
\vskip 5 pt

\indent We first consider the case when $x_{2}w \in E(G)$. So, $\{x, z, w\} \subseteq N(x_{1}) \cap N(x_{2})$. Because $|N(x_{1}) \cap N(x_{2})| \leq 3 = c_{1}$, $x_{1}$ and $x_{2}$ are adjacent to distinct vertices $y_{2}$ and $y_{3}$. Renaming vertices if necessary, $x_{1}y_{2}, x_{2}y_{3} \in E(G)$. Thus, $wy_{2}, wy_{3} \notin E(G)$ because $a = 0$. Therefore, $y_{2}$ can be adjacent only $x_{3}$ implying a contradiction that $deg(y_{3}) < 4$.
\vskip 5 pt

\indent We now consider the case when $x_{2}w \notin E(G)$. Thus, $x_{2}y_{2}, x_{2}y_{3} \in E(G)$ as $deg(x_{2}) = 4$. Similarly, $x_{1}$ is adjacent either $y_{2}$ or $y_{3}$. Renaming vertices if necessary, we let $x_{1}y_{2} \in E(G)$. This implies by $a = 0$ that $y_{2}w \notin E(G)$. Further, $y_{3}w \in E(G)$ as otherwise $|N(x_{1}) \cap N(y_{3})| = 0 < c_{3}$. This yields $y_{1}w, y_{1}x_{3} \in E(G)$ because $deg(y_{1}) = 4$. Thus $x_{3}$ is adjacent to both $y_{2}$ and $y_{3}$. Clearly, the graph $G$ is $G_{1}$ and this proves the lemma.
\end{proof}

To complete the section, we suppose first that there exists $v \in V(G) \setminus N[u]$ such that $|N(v) \cap N(u)| \geq 3$. Because $c_{1} = 3$, $|N(v) \cap N(u)| = 3$. We may let $N(v) \cap N(u) = \{x, y, z\}$ and $\{w\} = N(u) \setminus N(v)$. Because $k = 4$, $v$ is adjacent to one vertex in $V(G) \setminus N[u]$, $s$ say. We see that $sx, sy, sz \notin E(G)$ since $a = 0$. So $sw \in E(G)$ otherwise $|N(s) \cap N(u)| = 0 < c_{3}$. Clearly, $wx, wy, wz \notin E(G)$ and this implies that $s$ is a vertex in $V(G)$ such that $N(s)$ has two vertices $v, w$ such that $N(v) \cap N(w) = \{s\}$. By Lemma \ref{b2}, $G$ is $G_{1}$. Hence, we may assume that that every vertex in $V(G) \setminus N(u)$ has at most $2$ neighbours in $N(u)$. Because $N(u)$ is independent, there are $12$ edges from $N(u)$ to $V(G) \setminus N[u]$. Because there are $6$ vertices in $V(G) \setminus N[u]$. This implies that $|N(v) \cap N(u)| = 2$ for each $v \in V(G) \setminus N[u]$. Moreover, by Lemma \ref{b2}, we can assume that every vertex $p \in \{x, y, z, w\}$ shares at least one neighbour with every vertex $q \in \{x, y, z, w\} \setminus \{p\}$ in $G'$. Hence, we may relabel the $6$ vertices in $V(G) \setminus N[u]$ as $v_{xy}, v_{xz}, v_{xw}, v_{yz}, v_{yw}, v_{zw}$ where the vertex $v_{ij}$ is adjacent to $i$ and $j$ for each $i, j \in \{x, y, z, w\}$. Because $a = 0$, $v_{xy}$ is not adjacent to $v_{xz}, v_{xw}, v_{yz}, v_{yw}$. Therefore, $deg(v_{xy}) < 4$, a contradiction. These complete the proof of the Main Theorem.

%\midskip

\section{Concluding Remark}
Quasi-strongly regular graphs of grade two have been widely studied via the independence from vertex $u$ of $t_i(u)$. Many results of these graphs have been published including the characterization of graphs with some class of parameters. For the graphs of higher grade, we believe that the lack of the above independence as mentioned in Example 2 of Goldberg \cite{G} causes the problem more complex. To the best of our knowledge, there is no any characterization of quasi-strongly regular graphs of grade three has been considered yet. Since the graph in Example 2 of \cite{G} was presented, we consider a generic class of quasi-strongly regular graphs that this graph belongs to. Therefore, we found a bounds condition of the number of vertices $n$ for a given degree $k$ of these graphs and with this condition we can characterize the class of $SQSR(11, 4, 0; 3, 2, 1)$ and the class of $SQSR(12, 4, 0; 3, 2, 1)$. For the future works, one may consider another class of quasi-strongly regular graphs of grade higher than $2$ to analysis some general conditions and to find a new technique to study these types of graphs. Additionally, a construction of these such graphs may also be considered.

\section{Acknowledgements}
The first author is supported by Thailand Science Research and Innovation
(TSRI) Basic Research Fund: Fiscal year 2021 under project number 64A306000005.

\end{document}